\documentclass[smallextended]{svjour3}       
\smartqed
\RequirePackage[OT1]{fontenc}
\RequirePackage[numbers]{natbib}
\RequirePackage[colorlinks,citecolor=blue,urlcolor=blue]{hyperref}

\newcommand{\cum}{\operatorname{cum}}

\usepackage{amssymb, amsmath, natbib, graphicx, euscript, mathrsfs, enumerate, empheq,subcaption,tabularx}

\newcommand{\eqd}{\stackrel{Law}{=}}

\newcommand{\mS}{\mathcal{S}}

\newcommand{\sign}{\operatorname{sign}}

\newcommand{\T}{\mathcal{T}}

\newcommand{\E}{{\mathbb E}}

\renewcommand{\i}{\mathrm{i}}

\def\R{\mathbb{R}}

\newcommand{\mean}{\operatorname{mean}}

\renewcommand{\P}{{\mathbb P}}

\newcommand{\B}{\mathcal{B}}

\renewcommand{\kappa}{\varkappa}

\newcommand{\I}{{\mathbb I}}

\newtheorem{thm}{Theorem}

\newcommand{\vX}{\vec{X}}
\newcommand{\vx}{\vec{x}}
\newcommand{\vb}{\vec{b}}
\newcommand{\vu}{\vec{u}}
\newcommand{\vmu}{\vec{\mu}}
\newcommand{\vs}{\vec{s}}

\newcommand{\vy}{\vec{y}}
\newcommand{\vT}{\vec{\T}}
\newcommand{\vS}{\vec{S}}
\newcommand{\vQ}{\vec{Q}}
\newcommand{\vD}{\vec{D}}
\newcommand{\vH}{\vec{H}}
\newcommand{\vrho}{\vec{\rho}}
\newlength\fboxseph
\newlength\fboxsepv

\begin{document}
\title{Multivariate asset-pricing model based on subordinated stable processes
\thanks{The study has been funded by the Russian Academic Excellence Project "5-100".
}
}
\titlerunning{Multivariate asset-pricing model based on subordinated stable processes}

\author{
Vladimir Panov
and 
Evgenii Samarin
}

\authorrunning{V.Panov, E.Samarin} 

\institute{International Laboratory of Stochastic Analysis and its Applications\\National Research University Higher School of Economics\\
            Shabolovka, 26, Moscow,  119049 Russia\newline\\                         
	\email{vpanov@hse.ru, samarinwk@gmail.com}
}

\date{Received: \today}
\maketitle

\begin{abstract}
In this paper we consider a multidimensional time-changed stochastic process in the context of asset-pricing modelling.  The proposed model is constructed from  stable processes, and its construction is based  on two popular concepts - multivariate subordination and L{\'e}vy copulas. From theoretical point of view, our main result is Theorem~\ref{thm}, which yields a simulation method from the considered class of processes. Our empirical study shows that the model represents the correlation between assets quite well. Moreover, we provide some evidence that this model is more appropriate for describing stock prices than classical time-changed Brownian motion, at least if the cumulative amount of transactions  is used for a stochastic time change.  





\end{abstract}
\section{Introduction}
\label{intro}
Stochastic time change is a well-used tool for the construction of  probabilistic models which are able to  
represent the so-called stylised features of stock prices. 
From mathematical point of view, the main idea is to replace the deterministic time \(t\) of a  stochastic process \(L_t\) (usually - of a L{\'e}vy process)  by a nondecreasing nonnegative process \(\T(s)\).  As a result, one obtains a process \(X_s=L_{\T(s)}\), which is referred to as \textit{a time-changed process}.

The economical interpretation of this operation  is based on the idea that  the ``business'' time \(\T(s)\) may run faster than the physical time in some periods, for instance, when the amount of transactions is high, see Clark (\citeyear{Clark}), An{\'e} and Geman (\citeyear{AneGeman}), Veraart and Winkel (\citeyear{VW}).   For instance, An{\'e} and Geman (\citeyear{AneGeman}) show that the stock prices can be modelled by a time-changed process with \(L\) equal to the Brownian motion with drift, and \(\T\) equal to the cumulative number of trades till time \(s.\) This choice of \(\T\) is intuitively correct, and leads to a good understanding of the model, see Tauchen and Pitts (\citeyear{TP}) and Andersen (\citeyear{andersen}).

As for the choice of  a process \(L\) as a Brownian motion with or without drift, it is mainly based on the Monroe theorem   (Monroe, \citeyear{Monroe}), which  says that the class of time-changed Brownian motions coincides with the class of all semimartingales.  Nevertheless, the Monroe theorem assumes that the processes \(\T\) and \(L\) may be dependent, and this drawback makes the statistical analysis almost impossible. 
 Monroe's theorem is a significant theoretical fact, which shows the importance of the considered class of models, but 
it is  almost useless for financial modelling. 

It would be  natural  to consider more general class of processes for \(L\) - for instance, L{\'e}vy processes, but under the assumption that \(L\) and \(\T\) are independent.  Moreover, in this paper we will show that  the application of subordinated jump-type processes for asset-price modelling has several advantages in comparison with classical time-changed Brownian motions. Theoretically this can be explained by the observation that if the trajectories of the process \(L \) are discontinuous, then rapid changes  in  log-returns are made not only due to the jumps in  the number of trades (as in the time-changed Brownian motion), but also due to stochastic factors, which are incorporated in \(L.\)
In fact,   most rapid changes in prices  can be explained by factors, which are not related to a particular asset at all - e.g., tweets of the president of the U.S., terrorist activity, etc.   All of these stochastic factors can be included in the model by considering jump-type processes for \(L\). 

The class of subordinated L{\'e}vy processes is analytically tractable, and is able to reproduce the dynamics of financial time series, see Barndorff-Nielsen and Shiryaev (\citeyear{BNS}).  But, as it is shown in Figueroa-L{\'o}pez (\citeyear{FL}) as well as in Belomestny and Panov (\citeyear{panov2013d}), the statistical inference for this class of models is rather complicated. In this paper, we propose to restrict the admissible class of processes for \(L\) to the class of  \textit{\(\alpha-\) stable processes} \(S_{t}\), which are defined as L\'evy processes with the following property: for any \(a>0\) there exists some function \(b: \R_+ \to \R\) such that 
\begin{eqnarray}
\label{sat}
\bigl\{ S_{at} \bigr\}_{t \geq0}  \eqd  \bigl\{ a^{1/\alpha} 
S_t+b(t) \bigr\}_{t \geq0}.
\end{eqnarray}
with  \(\alpha \in (0,2]\).  The comprehensive study of this class is given in brilliant books by Samorodnitsky and Taqqu (\citeyear{ST}), Bertoin (\citeyear{Bertoin}), Sato (\citeyear{Sato}).


The application of stable processes for asset-pricing modelling was discussed in various papers, starting from classical works by Mandelbrot (\citeyear{man}),  Mandelbrot and Taylor (\citeyear{MT}). These models are able to represent many stylized facts of financial data, see Cont (\citeyear{cont}), and therefore can be used in a wide range of financial applications; the list of references is collected by Nolan (\citeyear{NolanBib}).  The only difficulty, which can arise by using the models based on stable processes, is that these models typically have infinite second moment. Nevertheless, this drawback is not crucial, see, e.g.,  
Mittnik and Rachev (\citeyear{MR}), 
Grabchak and Samorodnitsky (\citeyear{GS}).

An important question related to the time-changed processes  is how to generalise the idea of stochastic time change to the multivariate case. The first idea is to change the time \(t\) in all components of a L{\'e}vy process \(\vec{L}(t)= \left( L_1(t), ..., L_d(t) \right)\)  by the same stochastic process \(T(s)\). This approach is rather natural, but the resulting model is difficult to interpret in the context of price modelling, since the business time is different for different assets.  Another idea is to take \textit{a multivariate subordinator} - a L\'evy process \( \vec{\T}(s)=\left(\T_{1}(s), ...,\T_{d}(s) \right)\), such that each component is a one-dimensional subordinator, and change the time in \(L_i(t)\) by \(\T_i(s), \; i=1..d:\)
\begin{eqnarray}\label{MMM}
\vec{X}(s) = \left(
L_1 (\T_1(s)), ..., L_d (\T_d(s))
\right).
\end{eqnarray}
This concept, known as  the  multivariate subordination, was introduced by  Barndorff-Nielsen, Pedersen and Sato (\citeyear{BNPedSato}) for the case when the processes \(L_1,..., L_d\) are independent. In this paper, we consider the model \eqref{MMM} for the situation when \(\vec{L}(t)\) is a multivariate stable process (with possibly dependent components) and \(\vec{\T}(s)\) is a multivariate subordinator. Dependence between the components of \(\vec{\T}(s)\) is described via \textit{a  L{\'e}vy copula} - another popular concept in the theory of jump-type processes.

One of the contributions of this paper is a multivariate series representation for the considered class of processes, which yields a method for simulation from  the model. Our proof  is based on the paper by Rosi{\'n}sky (\citeyear{Rosinsky}), which has been already used for some previous results of this type, see Section~6.5 from the book by Cont and Tankov (\citeyear{ContTankov}). Nevertheless, to the best of our knowledge, these techniques were never applied to the models based on the  subordination of the multivariate stable processes. Simulation method developed in this paper allows to reproduce the processes with the same probabilistic structure as original process, and the simulated data can be efficiently used for testing various trading strategies.

The rest of the paper is organised as follows.  The next section is devoted to an overview of the most important properties of stable processes. In Section~\ref{1d}, we provide numerical example, which illustrates the advantages of this model in comparison with the  classical subordinated Brownian motion for the one-dimensional case. Section~\ref{dep} presents possible dependence structures for the multivariate models based on stable processes. In Section~\ref{MS}, we formulate some theoretical results related to the class of multivariate subordinated stable processes. In particular, Theorem~\ref{thm} yields a method for simulation from the considered model. In Section~\ref{ea}, we show an application of this result to stock returns, and also describe the estimation scheme for all parameters of the considered model. Appendix contains the proof of Theorem~\ref{thm}.

%
%
%

\section{ Stable processes}\label{stable2}

\subsection{Univariate stable processes}
In the one-dimensional case, the stable process \(S_t\)  can be parametrised by four real numbers.  In what follows, we will employ the notation \(\mS_{\alpha} (\sigma,
\beta, \mu)\), which means that 
the characteristic function of \(S_t\) can be represented in the following form
\begin{eqnarray}\label{eqq}
\phi_{S_t} (u) = \E
\left[
  e^{\i u S_t}
\right]
&=&
\exp\Bigl\{t 
\Bigl(
\i u \mu - 
\sigma^{\alpha}|u|^{\alpha}
\left(1 - \i \beta \sign(u) \theta(u)
\right) 
\Bigr) 
\Bigr\}, 
\\ \nonumber
\mbox{with} \qquad
\theta(u) &=&
  \begin{cases}
\tan\frac{\pi\alpha}{2}, &\text{if $\alpha\ne 1$;}\\
 - \frac{2}{\pi} \log|u|, &\text{if $\alpha = 1$},
\end{cases}
\end{eqnarray}
where \(\sigma \in \R_+\) is a scale parameter, \(\mu \in \R\) is a drift, and \(\beta \in [-1,1]\) is a skewness parameter. 
Note that in this parametrisation,  \eqref{sat} can be specified as 
\begin{eqnarray}\label{SSS}
\bigl\{ S_{at} \bigr\}_{t \geq0}  \eqd  \bigl\{ a^{1/\alpha} 
\left( S_t - \mu t \right) + a \mu t \bigr\}_{t \geq0},
\end{eqnarray}
provided \(\alpha \ne 1.\)

The L{\'e}vy measure of a stable process has a density equal to 
\begin{eqnarray}\label{Levyy}
s(x) = \frac{A}{x^{1+\alpha}} \cdot \I\left\{ x>0 \right\}
+
\frac{B}{|x|^{1+\alpha}} \cdot \I\left\{ x<0 \right\},
\end{eqnarray}
where \(A,B \geq 0\).  Therefore, the jump activity of the stable process  essentially depends on the parameter \(\alpha\), which coincides with the Blumenthal-Getoor index. Figure~\ref{Stable} illustrates the typical trajectories of the one-dimensional  \(\alpha-\)stable process depending on this parameter: when \(\alpha\) is close to 2, the process behaves similar to a Brownian motion, and when \(\alpha\) is close to 0, it looks like a compound Poisson process.

\begin{figure*}
\begin{center}
\includegraphics[width=0.75\linewidth ]{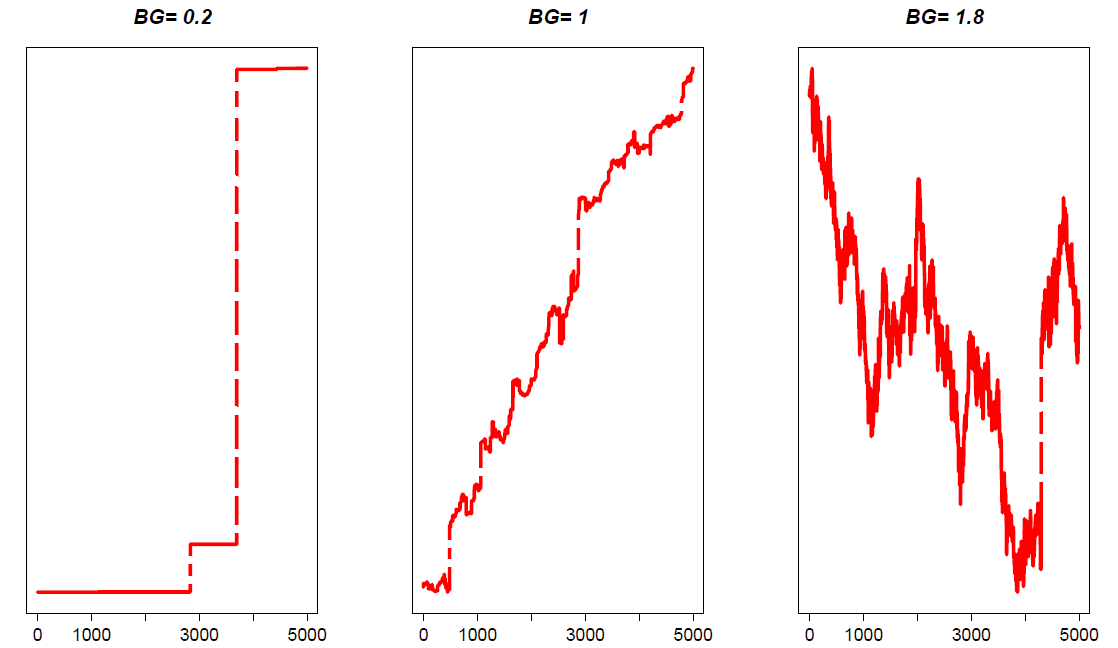}\caption{\label{Stable}Typical trajectories of an \(\alpha-\)stable process depending on the parameter \(\alpha.\)}
\label{fig1}
\end{center}
\end{figure*}

\subsection{Multivariate stable processes}
\label{MSP}
Multivariate stable processes \(\vS(t) \in \R^d\) can be defined as a L{\'e}vy  process such that \eqref{sat} is fulfilled with a function \(b: \R_+ \to \R^d\). The dependence between the components is described via \textit{the spectral measure} \(\Lambda\) on the unit  sphere \( \mathbb{S}_d\) in \(\R^d.\) A multidimensional analogue of \eqref{eqq} reads as 
 \begin{multline*}
\phi_{\vS(t)} (\vu) = \E
\left[
  e^{\i \langle \vu, \vS(t)\rangle}
\right]
\\=
 \exp
 \Bigl\{ t \Bigl(
  \i \langle\vu, \vmu\rangle 
  -\int_{ \mathbb{S}_d}
 |\langle\vu,\vs\rangle|^\alpha 
 \left(
 1 - \i \sign\left(
  \langle\vu,\vs\rangle
 \right) \theta \left(\langle\vu,\vs\rangle\right)
\right) \Lambda (d\vs)
\Bigr)\Bigr\},
 \end{multline*}
 where \(\vmu \in \R^d, \vu \in \R^d.\) For instance, if \(d=1\), then the measure \(\Lambda\) is supported on \(\{1,-1\},\) and the last formula reduces to \eqref{eqq} with \[\sigma= \Bigl( \Lambda(\{1\}) + \Lambda(\{-1\})\Bigr)^{1/\alpha}, \qquad \beta = \sigma^{-\alpha}\Bigl(\Lambda(\{1\}) -\Lambda(\{-1\})\Bigr).\] This measure determines also the L{\'e}vy measure, namely 
   \begin{equation}\label{eq:StablePrMeasure}
    \nu(B)=\int\limits_{\mathbb{S}_d}\int\limits_{\R_+}\frac{\mathbb{I}_{B}(r \vs)}{r^{1+\alpha}}\, dr \,\Lambda (d\vs)
  \end{equation}
  for all Borel sets $B\subset \R^{d}.$ 

Projection of a stable process to any vector \(\vu\in \R^d\) is  a univariate stable process \(\mS_{\alpha} (\sigma(\vu),
\beta(\vu), \mu(\vu))\), where the parameters \(\sigma(\vu),\beta(\vu), \mu(\vu) \)  are equal to 
 \begin{eqnarray*}
\sigma(\vu) &=&
\left(
\int_{\mathbb{S}_d} |\langle\vu, \vs\rangle|^\alpha \Lambda(d\vs)
\right)^{1/\alpha};\\
\beta(\vu) &=&\sigma(\vu)^{-\alpha} \int_{\mathbb{S}_d}|\langle\vu, \vs\rangle|^\alpha \sign(\langle\vu, \vs\rangle)\Lambda(d\vs);\\
\delta(\vu)&=&\begin{cases}\langle\vu, \vmu\rangle,& \alpha \ne 1\\\langle\vu, \vmu\rangle -\frac{2}{\pi}\int_{\mathbb{S}_d} \langle\vu, \vs\rangle \ln|\langle\vu, \vs\rangle|\Lambda(d\vs),&\alpha=1.
\end{cases}
 \end{eqnarray*}
Therefore, the spectral measure \(\Lambda\)   determines dependence structure among elements of the vector. 


\section{Asset-pricing in the one-dimensional case} \label{1d} 
Let us provide the following example. We consider  30-minutes Apple, Microsoft and GE (General Electric) stock prices traded on the Nasdaq over the period from October, 18, 2017, till May, 1,  2018. Denote for the interval number \(j=1,2,...,1658\),  the stock price by \(P_j\) and the number of trades by \(N_j\). We examine the model 
\begin{eqnarray}\label{AG}
\log \left(
  \frac{P_j}{P_0}
\right)
= 
S_{\cum(N_j)},
\end{eqnarray}
where \(S_t\) is an \(\alpha\)-stable process \(\mS_{\alpha} (\sigma,
\beta, \mu)\) and \(\cum(N_j)\) is a cumulative number of trades over the periods \(1, 2,.., j.\) Processes \(S_t\) and  \(\cum(N_j)\) are assumed to be independent.  In this study, we aim to estimate the parameters of the stable process \(S_t\)  and to show that  the choice  \(\alpha=2\) (corresponding to the model by An{\'e} and Geman \cite{AneGeman}) is not optimal. 

First, note that for any \(\alpha \in (1,2]\),
\begin{eqnarray*}
\E \left[\log \left(
  \frac{P_j}{P_{j-1}}
\right)
\right]
= 
\E\left[ S_ 1 \right] \cdot
\E \left[N_j\right] = \mu \E \left[N_j\right],
\end{eqnarray*}
and therefore the parameter \(\mu\)  can be estimated by
\begin{eqnarray*}
\hat{\mu}_n = \frac{
\mean\left(\log \left(
 P_j/P_{j-1}
\right)\right)
}{
\mean \left(
N_j
\right) 
},
\end{eqnarray*}
where \(\mean(\cdot)\) stands for the mean value of all available observations.

Second, we fit the parameters \(\alpha,\beta, \sigma\) using the following equality, which is a corollary from \eqref{SSS}:
\begin{eqnarray*}
\log \left(
  \frac{P_j}{P_{j-1}}
\right)
=
   Z_j N_j
^{1/\alpha} + \mu N_j,  \qquad j= 2, ..., 1658,
\end{eqnarray*}
where the values \(Z_2,..., Z_{1658}\) have the same distribution as  \((S_1-\mu).\) 
Taking different \(\alpha\) from the grid on \((1,2]\), we estimate \(\hat\beta_{\alpha}\) and \(\hat\sigma_{\alpha}\) from the data
\begin{eqnarray*}
\tilde{Z}_j = 
N_j^{-1/\alpha} 
\left( 
\log \left(
  \frac{P_j}{P_{j-1}}
\right)
-
\hat\mu_n N_j
\right),  \qquad j= 2, ..., 1658.
\end{eqnarray*}
 Finally, we choose the optimal values of the parameters by comparing the quality of the density estimators in terms of 
\begin{eqnarray*}
R(\alpha)=
\sum_{j=2}^{1658}
\left( p_\alpha (\tilde{Z}_j) -  \hat{p} (\tilde{Z}_j) \right)^2,
\end{eqnarray*}
where \(p_\alpha(x) \) is the density of the stable distribution \(\mS_{\alpha} (\hat\sigma_\alpha,
\hat\beta_\alpha, 0)\) and  \(\hat{p} (\cdot)\) is a kernel density estimator of the variables \(\tilde{Z}_2, \tilde{Z}_3,...\). 

The results can be visually analysed by the PP-plots and density plots in comparison with the results for the model of subordinated Brownian motion (\(\alpha=2\)), see Figures  \ref{fig:apl-ge} and \ref{fig:dens}. As can be seen on these figures, models with optimal values \(\alpha<2\) fit the data much better. 

Some technical remarks on this approach can be found in Section~\ref{ee4}. The numerical results of our estimation procedure are presented in Table~\ref{fig2}.

\begin{figure}
	\centering
	\includegraphics[width=0.75\linewidth]{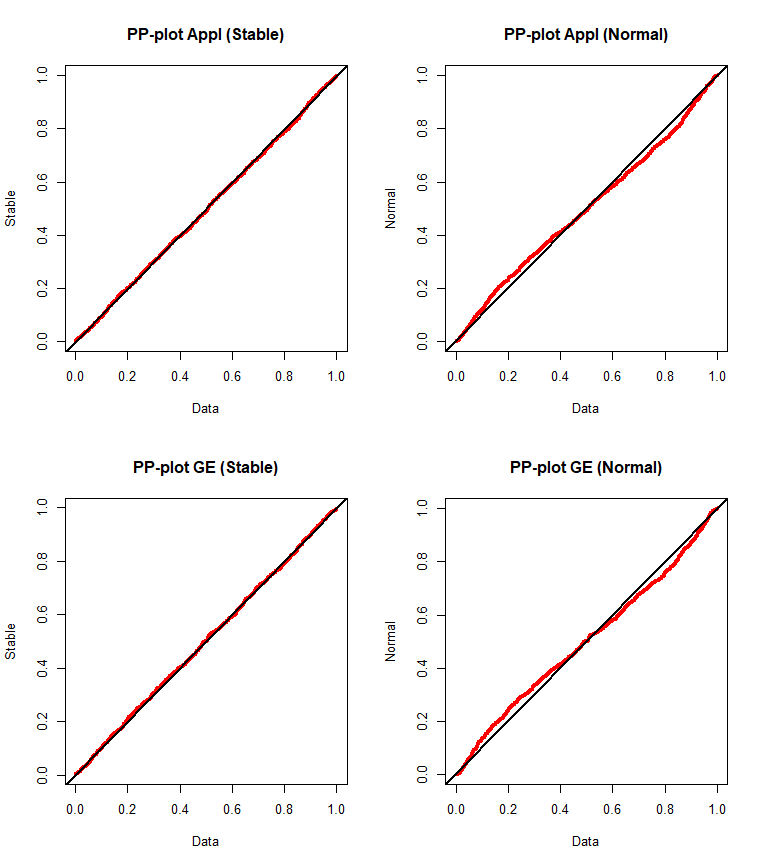}
	\caption{PP-plots for models of  subordinated stable processes  with \(\alpha<2\) (first column, \(\alpha=1.62\) for Apple and \(\alpha=1.83\) for GE) and subordinated Brownian motion (second column). Plots in the first row correspond to Apple stock prices, in the second row - to GE stock prices. }
	\label{fig:apl-ge}
\end{figure}

\begin{figure}
	\centering
	\includegraphics[width=0.75\linewidth]{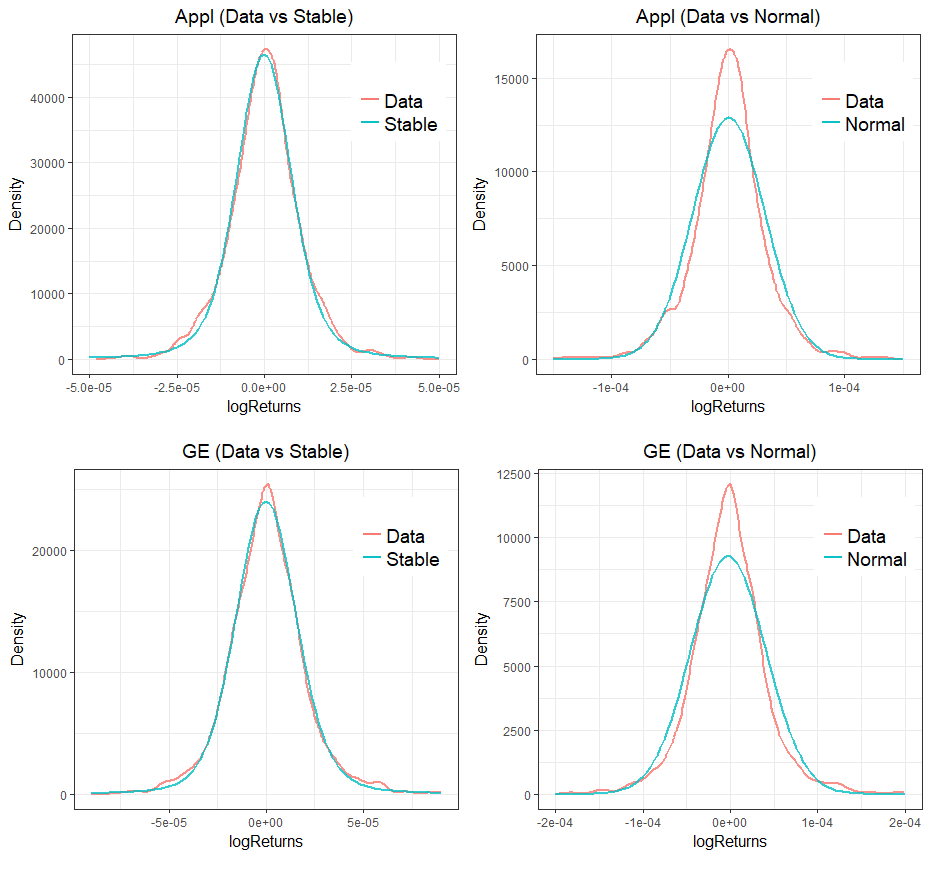}
	\caption{Density plots for models of  subordinated stable processes with \(\alpha<2\)  (first column, \(\alpha=1.62\) for Apple and \(\alpha=1.83\) for GE)  and subordinated Brownian motion (second column). Plots in the first row correspond to Apple stock prices, in the second row - to GE stock prices.}
	\label{fig:dens}
\end{figure}

\section{Dependence structures for stable processes and related models}
\label{dep} 
Description of the dependence structure between stochastic processes turns out to be  a rather complicated task, which cann't be solved by using only the copula approach. In fact, there is a common opinion, which was expressed by Thomas Mikosch in~\cite{Mikosch} as follows:  \textit{Copulas completely fail in describing complex space- time dependence structures}. Direct application of the copula approach to stochastic processes (that is, describing the dependence for any time moment) meets serious difficulties - for instance, it turns out that even in the simplest cases the copula also depends on \(t\), see \cite{Tankov}.

Therefore, one should introduce another object, which can describe the dependence in time-independent fashion. In the context of L{\'e}vy processes, a natural candidate is the L{\'e}vy copula.
\begin{definition} \label{defii}
A \(d\)-dimensional L{\'e}vy copula \(F\) is  a function from \(\bar{\R}^{d}\) to \(\bar{\R}\)  such that 
\begin{enumerate}
\item \(F\) is grounded, that is, \(F\left(\vec{u}\right) =0\) if \(u_{i}=0\) for at least one \(i=1,..,d\).
\item \(F\) is \(d\)-increasing.
\item \(F\) has uniform margins, that is, \(F^{(1)}(v) =...= F^{(d)}(v) =v\), where \[F^{(j)}(v) = \lim_{u_{1},...,u_{j-1},u_{j+1},u_{d} \to \infty} F\left(u_{1},...,u_{j-1},v,u_{j+1},...,u_{d}\right), \; j=1..d,\]
\item \(F\left(u_{1},..., u_{d}\right) \ne \infty \) for  \(\left( u_{1}, ..., u_{d} \right) \ne \left(\infty, ...,\infty\right)\).
\end{enumerate}
\end{definition}
Let us clarify the main aspects of the L{\'e}vy copula theory for the case of multivariate subordinators, that is, for the multivariate L{\'e}vy processes such that its components are non-decreasing (or, equivalently, non-negative). To do this, we need the notion of the tail integral,  which is defined for a process  \(\T\) from this class as
\begin{eqnarray*}
U\left(
  x_1,..., x_d
 \right) = \nu \left(
  \left[
  x_1, +\infty
  \right)
  \times...\times
    \left[
  x_d, +\infty
  \right)
\right), \qquad x_1,...,x_d > 0,
\end{eqnarray*}
where \(\nu\) is the L{\'e}vy measure of \(\T.\)
 An analogue of the Sklar theorem for ordinary copulas is the following statement:    for any multivariate subordinator \(\T\) with tail integral \(U\) and marginal tail integrals \(U_{1}, ..., U_{d}\), there exists a (positive) L{\'e}vy copula \(F\) such that  
\begin{eqnarray}
\label{Sklar}
	U (x_{1}, ..., x_{d}) = F \left( U_{1} (x_{1}), ...,  U_{d} (x_{d}) \right)
\end{eqnarray}
and vice versa, for any L{\'e}vy copula \(F\) and any one-dimensional L{\'e}vy process with tail integrals \(U_{1},  ..., U_{d}\)  there exists a \(d\)-dimensional L{\'e}vy process with tail integral \(U\) given by \eqref{Sklar} and marginal tail integrals \(U_{1}, ..., U_{d}\). This notion of tail integral as well as the Sklar theorem for L{\'e}vy copulas can be generalized for arbitrary L{\'e}vy processes, but its practical usage raises a lot of questions -  for instance, simulation of a multivariate L{\'e}vy process with given  L{\'e}vy copula is rather complicated, see \cite{ContTankov}. 


Returning to the multivariate stable processes,  it would be a worth mentioning that the dependence between components of a stable processes can be characterised via the spectral measure \(\Lambda\) on the unit sphere, see Section~\ref{stable2}. A crucial disadvantage of this method is that  this way of describing the dependence leads to a rather narrow class of models, because the resulting process in this case will be also stable. Different aspects of dependence structures for the multivariate stable distributions are discussed in Chapter~4 from \cite{ST}. 

In the next section we present a model, which is based both on spectral measures and L{\'e}vy copulas.


\section{Multivariate subordination of stable processes}\label{MS} 
Consider  a \(d-\)dimensional stable process \(\vS(t)=\left(S_{1}(t), ..., S_{d}(t)\right)\). Denote the characteristic exponent of the \(j-\)th component by 
\begin{eqnarray*}
\varphi_{S_j}(v) := \log \E\left[
  e^{\i  u S_j (1) }
\right] &=&
\i  u \mu_j
-\sigma_j^{\alpha_j}|u|^{\alpha_j}\left(1-i\beta_j\,\mathrm{sgn}(u)\,\theta(u_j) \right), \quad j=1..d,
\end{eqnarray*}
see Section~\ref{stable2} for notation. Denote the L{\'e}vy measure of \(S_j\)  by \(\nu_j(x)\).

Let  \(\vec{\T}(s)=\left(\T_{1}(s), ...,\T_{d}(s) \right)\) be a \(d\)-dimensional subordinator, that is, a L{\'e}vy process in \(\R^{d}\) such that its components \(\T_{1}, ..., \T_{d}\) are one-dimensional subordinators. Denote the Laplace exponent of the process \(\vT\) by
 \begin{eqnarray}\label{laplace}
\psi_{\T}(\vu) := \log \E\left[
  e^{\langle \vT (1), \vu \rangle}
\right]=
		 \langle \vrho, \vu \rangle  +
		\int_{\R} \left( e^{\langle \vu, \vx \rangle} -1 \right)\eta(d\vx),
\qquad \vu \in \R^d,	  
\end{eqnarray}
where  \(\vrho = \left(\rho_{1}, ...,\rho_{d}\right)\in \R^{d}\), \(\rho_{i} \geq 0,\;  i=1,...,d\), and  \(\eta\)  is a L{\'e}vy measure in \(\R_{+}^{d}\).  In other words, the L{\'e}vy triplet of \(\vT(s)\) under zero truncation function is \(\left(
  \vrho, 0, \eta
\right)\).

Assume  that for any \(i=1,..,d,\) \(\T_{i}(s)\) and \(S_{i}(s)\) are independent, and  define the multivariate subordinated process as
\begin{eqnarray}
\label{mainn}
 \vX (s) = \Bigl( X_{1}(s), ..., X_{d} (s) \Bigr) := \Bigl( S_{1}(\T_{1}(s)), ..., S_{d}(\T_{d}(s)) \Bigr).
\end{eqnarray}
 It would be a worth mentioning that \(\vX(s)\) is a L{\'e}vy processes. Its L{\'e}vy triplet can be represented via  the L{\'e}vy triplets of the processes \(\vS\) and \(\vT,\) see Appendix~\ref{app1}.

 Before we will formulate our main result, it would be important to note that for any \(v \geq 0\),  \(\widetilde{F}(u_1,..., u_{d-1}| v) = \partial F (u_1,..., u_{d-1}, v) / \partial v\) is a distribution function on \(\R_{+}^{d-1}\). If \(F\) is 2-dimensional, this result is given as  Lemma~5.3 in \cite{ContTankov}; the proof for the general, \(d-\)dimensional case, follows the same lines.
  \begin{thm}
 \label{thm}
Consider  the model~\eqref{mainn}, where  \(S_1,...S_d\) are  independent stable processes and  \(\vec{\T}(s)=\left(\T_{1}(s), ...,\T_{d}(s) \right)\) is a \(d\)-dimensional subordinator  (with possibly dependent components). Assume that 
the L{\'e}vy measure  \(\eta\) of the  process \(\T\) satisfies \[\int_{|\vx| \leq 1} |\vx|^{1/2} \eta(d\vx) < \infty.
\] and \(\vrho = 0\), see \eqref{laplace} for notation.

Denote by  \(F(u_{1},...,u_{d})\) a positive L{\'e}vy copula between \(\T_{1}(s), .., \T_{d}(s)\). Moreover, assume that  \begin{enumerate}
\item [(A1)] \(F(u_{1},...,u_{d})\) is continuous and the mixed derivative \(\partial^{d} F (u_{1},...,u_{d})/ \partial u_{1}... \partial u_{d}\) exists in \(\R_{+}^{d}\); in other words, the distribution function \(\widetilde{F}(u_1,..., u_{d-1}| v)\) is absolutely continuous for any \(v \geq 0;\)
\item [(A2)] there exist  functions \(h_1, ..., h_{d-1}: \R \times \R_+ \to \R\) and random variables \(\xi_1, ..., \xi_{d-1}\) such that 
\begin{eqnarray*}
\P\left\{
h_1 (\xi_1, v) \leq u_1, ..., 
h_{d-1} (\xi_{d-1}, v) \leq u_{d-1} 
\right\} 
= \widetilde{F}(u_1,..., u_{d-1}| v).
\end{eqnarray*} 
\end{enumerate}
Then 
	\[\vec{X}(s) \eqd  \vec{Z}(s), \qquad \forall s \in [0,1],\]	
where the  \(d\)-dimensional stochastic process \(\vec{Z}(s)= \left( Z_{1} (s), ...,  Z_{d}(s)\right)\) is defined as follows: 
\begin{eqnarray}
\nonumber
		Z_{k} (s)  &:=&  \sum_{i=1}^{\infty}  \left[  \left( G^{(k)}_{i} - \mu_i \right) \left(
		 U_{k}^{(-1)}\left(
	h_{k} (Q^{(k)}_{i}, \Gamma_{i} )
		\right)
\right)^{1/ \alpha_k}  
\right.\\&&
\left.
\hspace{3cm}
+
\mu_i U_{k}^{(-1)}\left(
	h_{k} (Q^{(k)}_{i}, \Gamma_{i} )
		\right) 
		\right]
 \I\left\{ R_{i} \leq s \right\} 
\end{eqnarray}
for \(k=1..(d-1)\), and 
\begin{eqnarray}
\nonumber
		Z_{d} (s)  &:=&  \sum_{i=1}^{\infty}  \left[  \left( G^{(d)}_{i} - \mu_i \right) \left(
		 U_{d}^{(-1)}\left(
\Gamma_i
		\right)
\right)^{1/ \alpha_k}  
+
\mu_i U_{d}^{(-1)}\left(
\Gamma_i
		\right) 
		\right]
 \I\left\{ R_{i} \leq s \right\},\\
 &&\label{res2}
\end{eqnarray}
and
 \begin{itemize}
 \item \(U_{1}, ..., U_{d}\)  are tail integrals of the subordinators \(T_{1}, ..., T_{d}\) resp.,  and \newline
 \(U_{1}^{(-1)}, ..., U_{d}^{(-1)}\) are their generalized inverse functions, that is,
 \begin{eqnarray*}
	U_{i}^{(-1)} (y) =  \inf \left\{ 
		x>0: \quad U_{i}(x) <y 
	\right\}, \quad i=1..d, \quad y \in \R_{+};
\end{eqnarray*}
\item 
\(\Gamma_{i}\) is a sequence of jump times of a standard Poisson process;
\item 
\(R_{i}\) is a sequence of i.i.d. r.v.'s, uniformly distributed on \([0,1]\);
\item for any \(i=1,2,..\), \(G^{(1)}_{i}, ..., \; G^{(d)}_{i} \) - are independent stable random variables, \(G_i^{(j)} \sim S_{\alpha_j}(\beta_j, \sigma_j, 0).\)
 \item for any \(i=1,2,..\),  \(\vQ_{i}:=\left( Q^{(1)}_{i}, .., Q^{(d-1)}_{i} \right) \)- sequence of i.i.d. random vectors with the same distribution as  \(\left(
  \xi_1, ..., \xi_{d-1} 
\right)\),  
 \end{itemize}
 and all sequences \(\Gamma_i, R_i,G^{(1)}_{i}, ..., \; G^{(d)}_{i}, \vQ_{i} \)  are jointly independent.    
  \end{thm}
\begin{proof}
The proof is given in Appendix~\ref{proof2}.
\end{proof}

\begin{remark} Some conditions, which guarantee that (A2)  holds, can be found in \cite{panov2017a}. For instance, it is sufficient to assume that there exists a density function \(p^{*}: \R_+^{d-1} \to \R\) and \((d-1)\) functions \(f_{j}^{*}: \R_{+}^{2} \to \R_{+}, \; j=1..(d-1)\), such that 
\begin{enumerate}
\item for any \(u,x>0\),
\begin{eqnarray}
\nonumber
	\int_{-\infty}^{f_{1}^{*}(u_{1}, v)}  ... \int_{-\infty}^{f_{d-1}^{*}(u_{d-1}, v)} p^{*}(z_{1}, ..., z_{d-1}) dz_{d-1} ...dz_{1}  &=&\widetilde{F}(u_1,..., u_{d-1}| v);\\
&&	\label{fp}
\end{eqnarray}
\item the functions \(f_{j}^{*}(u_{j}, v), j=1..(d-1)\) monotonically increase in \(u_{j}\) for any fixed \(v\), and moreover, for any  \(j=1..(d-1)\) and any \(y>0\), the equation
\begin{eqnarray*}
	f_{j}^{*}(u_{j},v) = y
\end{eqnarray*}
 has a  closed-form solution with respect to \(u_{j}\);  we denote this solution by \(h_{j}^{*}(y, v)\).
  \end{enumerate}
 In fact, in this case 
\begin{multline*}
	\frac{\partial^{d-1}  \widetilde{F}(u_1,..., u_{d-1}| v).}{\partial u_{1} ...\partial u_{d-1} }
	\\=
	\frac{ \partial f_{1}^{*}(u_{1}, v)} {\partial u_{1}}
	...
	\frac{ \partial f_{d-1}^{*}(u_{d-1}, v)} {\partial u_{d-1}}
	\cdot
	p^{*} \left( 
		f_{1}^{*}(u_{1}, v),
		...,
		f_{d-1}^{*}(u_{d-1}, v)
	\right),
\end{multline*}
and therefore (A2) is fulfilled with \(h_j = h_j^*,\; j=1..d,\) and r.v.'s \((\xi_1, ..., \xi_{d-1})\) having distribution with probability distribution function \(p^*.\)

Note that the  conditions (A1)-(A2) are fulfilled for the Clayton-L{\'e}vy copula, 
\begin{eqnarray*}
	F_{C}(u_{1}, ..., u_{d}) = (u_{1}^{-\theta}+...+ u_{d}^{-\theta})^{-1/\theta}
\end{eqnarray*}
with some \(\theta>0\), 
as well as for any sufficiently smooth homogeneous L{\'e}vy copulas and mixtures of them, see \cite{panov2017a}, Section~5, Examples 1-3.

\begin{remark}
Theorem~\ref{thm} deals with the situation when the components of the stable process \(\vec{S}(t)\) are assumed to be independent. In more general case, when the spectral measure is concentrated on a finite amount of points, one can  employ the fact that this process  is in fact a linear transformation of independent stable processes. More precisely, if 
\begin{eqnarray*}
\Lambda (\cdot) = \sum_{j=1}^M \lambda_j \I_{\{\cdot\}} \left( \vec{s}_j \right)
\end{eqnarray*}
with \(\lambda_1,..., \lambda_M >0\) and \(\vec{s}_1, ...\vec{s}_M \in \R^d,\)
then 
\begin{eqnarray*}
\vec{S}(1) \eqd \sum_{j=1}^M \lambda_j^{1/\alpha} \xi_j \vec{s}_j + \vec\delta, 
\end{eqnarray*}
where \(\xi_1, ..., \xi_M\) are  independent one-dimensional stable random variables \(\mS_{\alpha} (1,
1, 0)\) and    \(\vec\delta \in \R^d,\)
see \cite{NNolan} and  Proposition~2.3.7 from \cite{ST}.

\end{remark}
\section{Empirical analysis}
\label{ea}

In this chapter, we consider the following model:

\begin{eqnarray}
\label{main}
\vX (s) = \Bigl( X_{1}(s), X_{2} (s) \Bigr) := \Bigl( S_{1}(\T_{1}(s)), S_{2}(\T_{2}(s)) \Bigr),
\end{eqnarray}
where $S_{1} (t),S_{2} (t)$ are two independent stable processes and $(\T_{1} (s),\T_{2} (s))$ is a 
two-dimensional subordinator. Dependence structure between $\T_{1} (s)$  and $\T_{2} (s)$ is described via the Clayton-L{\'e}vy copula:
\begin{eqnarray}
\label{clcopula}
F (x_{1},x_{2};\delta) = \Bigl( x_{1}^{-\delta}+x_{2}^{-\delta} \Bigr)^{-1/\delta}
\end{eqnarray}
Our empirical analysis consists in two stages:
\begin{enumerate}
	\item estimation of the parameters of the model:
	\begin{enumerate}
\item	estimation of the parameters of L{\'e}vy copula between  $\T_{1} (s)\) and \(\T_{2} (s)\):
\item	estimation of the parameters of stable processes $S_{1} (t)$ and $S_{2} (t)$;
\end{enumerate}
\item	simulation of the process with considered structure \eqref{main}.
\end{enumerate}
We apply this model to the real data of Apple, Microsoft and GE 30-minute returns over
 the period October, 18, 2017, till May, 1,  2018. For each 30-minutes  period, 
the value of return
 and number of trades are known. For each day, we ignore the first 30-minutes period because of 
 the abnormally low number of trades. 
 
We analyse 2  pairs of returns: Apple and Microsoft asset prices (highly correlated - namely, correlation coefficient is equal to 0.57),  Apple and General Electrics (correlation coefficient is small, equal to 0.1).   
\subsection{L\'evy copula estimation}

In what follows, we assume that subordinators are in fact compound Poisson processes (CPP) with positive jumps:

\begin{eqnarray*}
\T_{1} (s) = \sum\limits _{i=1}^{N_{1} (t)} X_{i},\:\:\:\:\:\:\:\:\:
\T_{2} (s) = \sum\limits _{j=1}^{N_{2} (t)} Y_{j},  
\end{eqnarray*}
where $X_{1}, X_2,...$ and $Y_{1}, Y_2, ...$ are i.i.d random variables having log-normal distribution with parameters $\mu_{1},\sigma_{1}$ and $\mu_{2},  \sigma_{2}$. $N_{1} (t)$ and $ N_{2} (t)$ are the Poisson processes with intensities $\lambda_{1}$ and $\lambda_{2}$ resp. The choice of the log-normal distribution for  jump sizes is verified by visual comparison of densities (see  Figure \ref{fig:qq_pp_jumps}).

\begin{figure}
\centering
\includegraphics[width=0.5\linewidth]{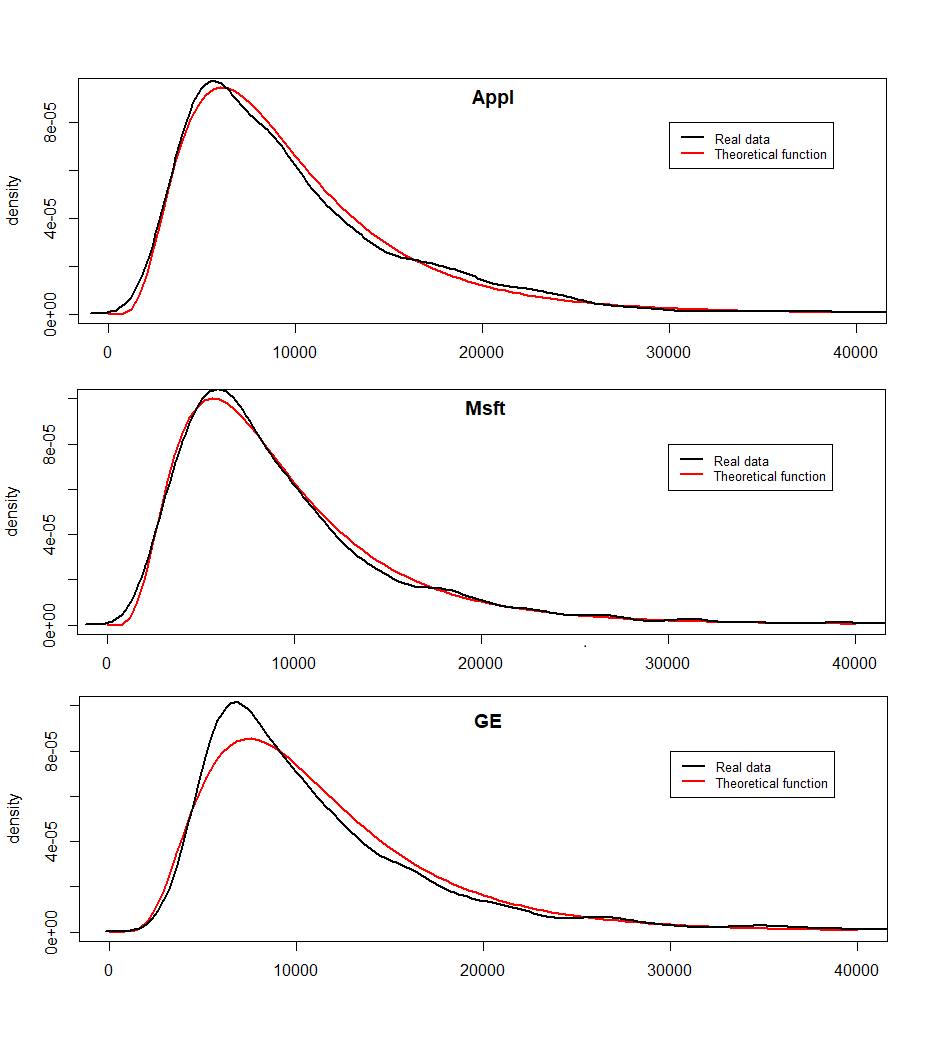}
\caption{Density functions of the number of trades (red lines) and the fitted log-normal distribution (red lines). The plots correspond to the Apple, Microsoft and GE asset prices. } 
\label{fig:qq_pp_jumps}
\end{figure}

Given that jumps occur at each moment for both components, the likelihood function for such process $(\T_{1} (s),\T_{2} (s))$ is equal to

\begin{multline}
L (\lambda_{1},\lambda_{2},\mu_{1},\mu_{2}, \sigma_{1}, \sigma_{2},\delta)=  
\biggl(\frac{ (1+\delta)(\lambda_{1}\lambda_{2})^{\delta+1}  }{4^{\delta} \sigma_{1} \sigma_{2}2\pi} \biggr)^{n}
\prod\limits_{i=1}^{n}x_{i}^{-1}y_{i}^{-1}\\
\cdot\exp\left\{
- \lambda^{\shortparallel}T-
0.5\sum\limits _{i=1}^{n}\Bigl(\frac{\ln x_{i}-\mu_{1}}{\sigma_{1}}\Bigr)^{2}-
0.5\sum\limits _{i=1}^{n}\Bigl(\frac{\ln y_{i}-\mu_{2}}{\sigma_{2}}\Bigr)^{2}
\right\}\\
\cdot\prod\limits_{i=1}^{n}\Big(\big[0.5\lambda_{1}(1-Erf_{1}(x_{i}))\big]^{\delta}+
\big[0.5\lambda_{2}(1-Erf_{2}(y_{i}))\big]^{\delta}\Big)^{-1/\delta-2}\\
\cdot \prod\limits_{i=1}^{n}\Big[(1-Erf_{1}(x_{i}))(1-Erf_{2}(y_{i}))\Big]^{\delta}
\end{multline}
where $x_{i}$ and $y_{i}, i = 1 \dots n$ are jumps of the first and the second components occurring up to some fixed time $T=n/4$ and $\lambda^{\shortparallel}=F(\lambda_{1},\lambda_{2},\delta)$, see \cite{Esm}. The results of numerical maximization of the function
$L$  are presented in Table~\ref{tab:MLE1}. Higher dependence between Apple and Microsoft stock prices is reflected by higher values of the parameter \(\delta\) (1.92, while for the second pair this parameter is equal to 0.8).

The results of the application of the Kolmogorov-Smirnov tests are presented in Table~\ref{tab:KS}.
The hypothesis about  the distribution is not rejected (at the 0.05 level of significance) for the assets in both pairs. 
 We arrive at the conclusion that the proposed model fits well the cumulative number of trades.
\begin {table}[htb]
\caption {MLE for the parameters of copula and marginal distributions} 
\label{tab:MLE1} 
\begin{center}
\noindent\begin{tabular}{|c|c|c|c|c|c|c|c|}
\hline 
\multicolumn{8}{|c|}{Estimated parameters} \\ 
\hline 
Pair         & $\mu_{1}$ &$\mu_{2} $ &$\delta $  &$  \sigma_{1}$ &$  \sigma_{2}$ &$\lambda_{1}$&$\lambda_{2}$\\ 
\hline 
Apple \&Msft & 8.82     & 8.01       & 1.92     &0.73         &  0.91      &     5.22       &  7.8      \\ 
\hline 
Appl  \& GE	 & 8.74     & 9.02      & 0.80       & 0.60        & 0.58      &      5.13      &  4.9    \\ 
\hline

\end{tabular} 
\end{center}
\end {table}

\begin {table}[htb]
\caption {Statistics of the Kolmogorov-Smirnov test: \(D_1, p_1\) - corresponding to the first asset in the pair, \(D_2, p_2\) - to the second.} 
\label{tab:KS} 
\begin{center}
\noindent\begin{tabular}{|c|c|c|c|c|}
\hline 
\multicolumn{5}{|c|}{Estimated parameters} \\ 
\hline 
Pair         & $D_{1}$ &$p-value_{1}$  & $D_{2}$  &$p-value_{2}$ \\ 
\hline 
Apple \&Msft & 0.0627     &  0.13       &  0.0466     & 0.44            \\ 
\hline 
Apple  \& GE	 & 0.0613     & 0.15     & 0.065       & 0.11          \\ 
\hline

\end{tabular} 
\end{center}
\end {table}

\subsection{Estimation of the parameters of stable process }
\label{ee4}
Estimation scheme for the parameters of the processes $S_{1} (t)$ and $S_{2} (t)$ is described in Section~\ref{1d}. We examine the model \eqref{AG} separately for each stock asset. Below we provide some technical details about the estimation procedure.

Estimation of the parameters is done by  the QMLE approach. 
According to the method described in \cite{nolan1997},  density of the stable random variable for the case $\alpha\neq1$, $\sigma=1$ and $\mu=0$ can be represented in the following form:
\begin{eqnarray*}
f(x;\alpha,\beta)&=&\frac{\alpha(x-\zeta)^{\frac{1}{\alpha-1}}}{\pi|\alpha-1|}
\int_{-\theta_0}^{\pi/2} V(\theta)  \exp\Bigl\{-(x-\zeta)^{\frac{\alpha}{\alpha-1}} V(\theta)\Bigr\} d\theta,  \qquad
 \text{if } x > \zeta, \\
f(x;\alpha,\beta)&=&f(-x;\alpha,-\beta),  \qquad \text{if } x < \zeta, \\
f( \zeta;\alpha,\beta)&=&\frac{\Gamma(1+\alpha^{-1})\cos(\theta_{0})} {\pi (1+\zeta^{2})^{1/(2\alpha)}}, \qquad \text{if } x = \zeta,
 \end{eqnarray*}
where 
\begin{eqnarray*}
\zeta &=&\zeta(\alpha,\beta)=-\beta \tan\left(\frac{\pi\alpha}{2}\right),\\
\theta_{0}&=& \theta_{0}(\alpha,\beta)=\frac{1}{\alpha} \arctan\left(\beta \tan\left(\frac{\pi\alpha}{2}\right) \right),\\
V(\theta) &=& V(\theta,\alpha,\beta)=(\cos\alpha\theta_{0})^{\frac{1}{a-1}}\left(\frac{\cos\theta}{\sin\alpha(\theta_{0}+\theta)}\right)^{\frac{1}{a-1}}\frac{\cos(\alpha\theta_{0}+(\alpha-1)\theta)}{\cos\theta}.
\end{eqnarray*}
Probability density function for the case when  $\sigma\neq1$ and $\mu\neq0$  can be calculated by the standardisation of a random variable. 
Estimation procedure of the parameters for stable distributions was conducted with the help of the R package "Stabledist" . Results of the numerical optimisation using this method are presented in Table~\ref{fig2}.

\begin{table}[htb]
\caption {\label{fig2} Estimated values of the parameters of stable distributions} 
\label{tab:MLE} 

\centering
\begin{tabularx}{\textwidth}{|X|X|X|X|X|}
\hline 
Stock	& $\alpha$  & $\beta$  & $\sigma$ & $\mu$ \\ 
\hline 
Apple	& 1.62 & 0.09 & 1.83e-05  & 3.02e-09  \\ 
\hline 
Microsoft	& 1.64 &0.15  & 2.10e-05   & 1.216e-08   \\ 
\hline 
GE	& 1.83 & 0.21 & 2.52e-05  &  -2.45e-08\\
\hline 
\end{tabularx} 
\end{table}

\subsection{Simulation}

Simulation algorithm is described below.

\begin{enumerate}
	\item  Simulate \(N\) i.i.d. standard exponential random variables $T_{j}$ , \(i = 1,\dots, N.\)
	\item  Simulate \(N\) independent 2-dimensional stable random variables  \((G_i^{(1)}, G_i^{(2)})\) with  \(G_i^{(1)} \sim S_{\alpha_1}(\beta_1, \sigma_1, 0),\) \(G_i^{(2)} \sim S_{\alpha_2}(\beta_2, \sigma_2, 0)\).
	\item Simulate N independent uniform random variables $R_{i}$ on [0, 1], \(i = 1,\dots, N.\)
    \item Simulate N independent random variables $Q_{i}$ with distribution function
     $H (z) =(z^{-\delta} + 1)^{−(1+\delta)/\delta}$ by the method of inverse function.
     \item Simulate a multivariate subordinated stable processes by (truncated) series representation:
     \begin{eqnarray*}
     Z_{1} (s)  &:=&  \sum_{i=1}^{N}  \left[  \left( G^{(1)}_{i} - \mu_1 \right) \left(
     U_{1}^{(-1)}\left(
     h_{1} (Q^{(1)}_{i}, \Gamma_{i} )
     \right)
     \right)^{1/ \alpha_1}  
     \right.\\&&
     \left.
     \hspace{3cm}
     +
     \mu_1 U_{1}^{(-1)}\left(
     h_{1} (Q^{(1)}_{i}, \Gamma_{i} )
     \right) 
     \right]
     \I\left\{ R_{i} \leq s \right\}, \\
     Z_{2} (s)  &:=&  \sum_{i=1}^{N}  \left[  \left( G^{(2)}_{i} - \mu_2 \right) \left(
     U_{2}^{(-1)}\left(
     \Gamma_i
     \right)
     \right)^{1/ \alpha_1}  
     +
     \mu_2 U_{2}^{(-1)}\left(
     \Gamma_i
     \right) 
     \right]
     \I\left\{ R_{i} \leq s \right\}.
     \end{eqnarray*}
  \subsection{Discussion} 
    
\;\;Table~\ref{tab:moments} contains the empirical confidence intervals for mathematical expectations, variances and the correlation coefficients based on 100 simulated trajectories. It turns out that the correlations are very well represented by the proposed model. In fact, confidence intervals are rather small and contain the true correlation parameter. For instance, for the pair Apple-Microsoft the correlation coefficient between log-returns is equal to 0.57, and the constructed confidence interval is [0.52, 0.57]. Moreover, the correlation between asset returns can be easily seen from Figure \ref{fig:msft-apple}, which represents 1 trajectory of Apple and Microsoft returns: after time moment 500 the prices are strongly correlated. Analogously, one can analyse the  Apple-GE pair. 
 As can be seen from Figure~\ref{fig:appl-ge},  dependence between simulated returns for  this pair  is much weaker than for Apple-GE. It is also reflected by the larger value of the parameter $\delta$ for the first pair. 

\;\;Therefore, we arrive at the conclusion that the proposed model can represent the dependence in terms of the correlation coefficients between asset returns both for the cases of relatively high correlation (as Apple-Microsoft) and small correlation (as Apple-GE).
    
    \begin{table}[htb]
\caption {Moments of simulated and real data} 
\label{tab:moments} 
\centering
\begin{tabularx}{\textwidth}{|X|X|X|X|X|X|}
\hline 
\multicolumn{6}{|c|}{Appl  Msft} \\ 
\hline 
& $E[ Z_{1}]\cdot10^{5}$&$E[ Z_{2}]\cdot10^{5}$ &$st.d.[ Z_{1}]\cdot10^{3}$  & $st.d.[ Z_{2}]\cdot10^{3}$&$cor[ Z_{1} Z_{2}]$ \\
\hline 
Simulated (95\% int.) & [1.1;4.5] & [9.0;12.1] & [2.21;4.31]  &  [3.13;5.47] & [0.52;0.59] \\ 
\hline 
Real            & 3.09         & 11.7        & 3.93         &   4.64        & 0.57 \\ 
\hline 
\multicolumn{6}{|c|}{Appl  GE}\\
\hline 
& $E[ Z_{1}]\cdot10^{5}$&$E[ Z_{2}]\cdot10^{5}$ &$st.d.[Z_{1}^{2}]\cdot10^{3} $  & $st.d.[Z_{2}]\cdot10^{3}$&$cor[Z_{1} Z_{2}]$ \\
\hline 
Simulated (95\% int.)	& [1.3;4.2]    & [-47.1;-26.5] & [1.93;3.95]  & [3.77;6.06]& [0.02;0.12] \\ 
\hline 
Real	            &      3.09    &     -28.9      & 3.93         & 5.68       & 0.10 \\ \hline 
\end{tabularx} 
\end{table}

    \begin{figure}
    	\begin{center}
    	\begin{subfigure}{1\linewidth}
   \includegraphics[width=0.7\linewidth]{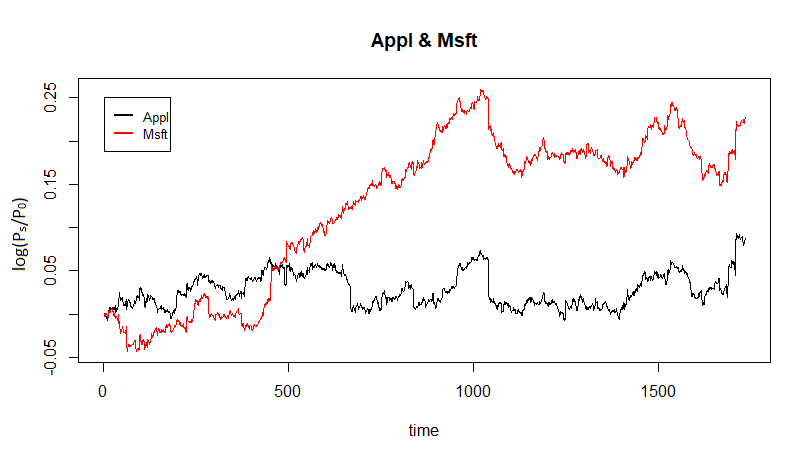}
    	\caption{}
    	\label{fig:msft-apple}
        \end{subfigure}
        \begin{subfigure}{1\linewidth}    	\includegraphics[width=0.7\linewidth]{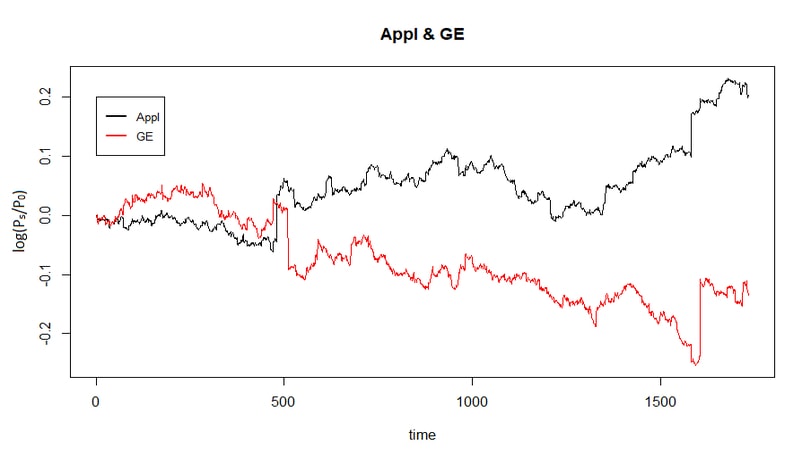}
    	\caption{}
    	\label{fig:appl-ge}
         \end{subfigure}
         \caption{Trajectories of the simulated processes for Apple and Microsoft asset prices (a)  and Apple and GE asset prices (b)}
         \end{center}
     \end{figure}

\end{enumerate}


\end{remark}

\appendix
\section{Theoretical properties of the subordinated stable process}\label{app1}
The next proposition reveals the relation between the L{\'e}vy triplet of the process \(\vX(s)\) and the L{\'e}vy triplets of the processes \(S_1,...S_d\) and \(\vT.\) 

 \begin{proposition}
 \label{thm2}
 \begin{enumerate}
\item 
 Let  \(\vX\) be a process defined by \eqref{mainn}, and assume that the one-dimensional processes \(S_1,..., S_d\) are independent.   Then  \(\vX\) is a \(d\)-dimensional L{\'e}vy process with the  characteristic function  equal to
\begin{eqnarray}
\label{cf}
\phi_{\vX(s)}(\vu) = 
\exp \left\{ 
	t \psi_{\T(s)}\left(
	   \varphi_{S_1} (u_1),
		...,
	  \varphi_{S_d} (u_d)
	\right)
\right\},
\end{eqnarray}
where \(\vu = \left(
  u_{1},...,u_{d}
\right).\) The L{\'e}vy triplet of the process \(\vX\) is equal to 
\(\left(
\vb, 0, \nu
\right)
\), where 
\begin{eqnarray*}
\vb &=& \int_{\R_{+}^{d}}\eta(d \vec{y}) \int_{|\vx| \leq 1} \vx \; m(d\vx, \vec{y})
+ 
\left(
  \rho_1 \mu_1,  ...,   \rho_d \mu_d
\right)^\top, \\
\nu(B) &=& \int_{\R_{+}^{d}}m \Bigl(B; \; \vy \Bigr) \; \eta(d \vy) \\
&&\hspace{1cm}
+
\int_B \left(
   \rho_1 \I_{A_1}(x_1) \nu_1 (dx_1)+..+
   \rho_d \I_{A_d}(x_d) \nu_d (dx_d)
\right),\\
&&\hspace{7cm}\qquad B \subset \R^{d}.
\end{eqnarray*}
where \(m(\cdot, \vec{s}) \) with \(\vec{s} = \left(s_{1},...,s_{d}\right)\) stands for the distribution of the random vector \(\left(S_{1}(s_{1}),..., S_{d}(s_{d})\right), \)  
and \(A_j\) is the \(j-\)th coordinate axis on \(\R^d,\) \(j=1..d.\)
\item If we additionally assume that the L{\'e}vy measure \(\eta\) of the  process \(\T\) satisfies \[
\int_{|\vx| \leq 1} |\vx|^{1/2} \eta(d\vx) < \infty.
\] and \(\vrho = 0\), then the process \(\vX(s)\) has bounded variation, and its characteristic function can be represented 
as 
\begin{eqnarray*}
\phi_{\vX(s)}(\vu)  = \exp\left\{
s \int_{\R} \left( e^{\langle \vu, \vx \rangle} -1 \right)\nu(d\vx)
\right\} 
\end{eqnarray*}
with 
\begin{eqnarray*}
\nu(B)= \int_{\R_{+}^{d}}m \Bigl(B; \; \vy \Bigr) \; \eta(d \vy),  \qquad B \subset \R^{d}.
\end{eqnarray*}
\end{enumerate}
 \end{proposition}
 \begin{proof}
 This  proposition follows from Theorem~3.3  in \cite{BNPedSato}.
 \end{proof}

\section{Proof of Theorem~\ref{thm}}
\label{proof2}
\textbf{1.} 
In the core of this proof lies the result by Rosi{\'n}sky \cite{Rosinsky}, 
 which we formulate in the simplified form below. 
 \begin{lemma} Assume that there exists a (multidimensional) random variable \(\vD\) in a measurable space \(S\) and  a
function \(\vH: \R_+\times S \to \R^{d}\) such that 
\begin{eqnarray}
\label{nu}
\nu(B) &:=& \int_{\R_{+}} \int_B
\breve{p}_{r}(\vx) d\vx\;
  dr, \qquad \qquad   r>0, \; B \in \B(\R^{d}),
\end{eqnarray}
is a L{\'e}vy measure, 
where \(\breve{p}_r(\vx)\) is a density function of the r.v. \(\vH(r,\vD).\)
Then the series
\(
X(s) = 
\sum_{i=1}^{\infty}
\vH\left(\Gamma_{i}, \vD_{i} \right) \cdot  \I \left\{
	R_{i} \leq s
\right\}
\), where
\begin{itemize}
\item \(\Gamma_{i}\)  is a sequence of jump times of a standard Poisson process,
\item \(\vD_i\) - sequence of i.i.d. r.v.'s with the same distribution as \(\vD,\)
\item \(R_{i}\) - sequence of i.i.d. r.v.'s uniformly distributed on \([0,1]\), 
\end{itemize}
converges almost surely and uniformly on \(t \in [0,1]\) to a L{\'e}vy process with triplet \((\vb, 0, \nu),\) where 
\begin{eqnarray}
\label{mu}
\vb = \int_{\R_+} \int_{|\vx| \leq 1} \vx \breve{p}_{r}(\vx) d\vx \; dr,
\end{eqnarray}
provided that the last integral exists.
\end{lemma}

 
\textbf{2.} In what follows, we denote by \(\nu\) the L{\'e}vy measure of the process 
\begin{eqnarray*}
 \vX^\circ (s) :=\Bigl( S^\circ_{1}(\T_{1}(s)), ..., S^\circ_{d}(\T_{d}(s)) \Bigr).
\end{eqnarray*}
As it is shown in Proposition~\ref{thm2}, 
\begin{eqnarray*}
\nu(B) &=& \int_{\R_{+}^{d}}\mu \Bigl(B; \; \vy \Bigr) \; \eta(d \vy), \qquad 
B \in \B(\R^{d}).
\end{eqnarray*}
Applying  Proposition~5.8 from \cite{ContTankov}, we conclude that 
\begin{eqnarray}
\nonumber 
 \nu (B) &=& \int_{\R_{+}^{d}}   \mu\Bigl(B \; ; \vy  \Bigr) \; 
 \left.\frac{\partial^{d} F}{\partial u_{1} \; ...\; \partial u_{d}}\right|_{\substack {u_{1}= U_{1}(y_{1}) \\ ... \\ u_{d}=U_{d}(y_{d})}} \;
 d\left( U_{1}(y_{1})\right)
 ...
 d\left( U_{d}(y_{d})\right) \\
 \label{mumu}
 &=& 
 \int_{\R_{+}^{d}}   \mu\Bigl( B \; ;  \vec{U^{-1}} (\vu)\Bigr) \; 
\frac{\partial^{d} F (\vu)}{\partial u_{1} \; ...\; \partial u_{d}} du_{1} ...du_{d},
 \end{eqnarray}
 where \(\vec{U^{-1}}(\vu) = \left( U_{1}^{-1} (u_{1}) ,..., U_{d}^{-1} (u_{d})\right)\).
In what follows, we consider the sets \(B = B_{1} \times ... \times B_{d}\), where \(B_{k} = [x_{k}, \infty), \; x_{k} \in \R, \; k=1..d\). For such \(B\), 
 \begin{eqnarray*}
\mu\Bigl( B; \;\vec{U^{-1}}(\vu)\Bigr) = \mu_1\Bigl( B_{1}; \; U_{1}^{-1} (u_{1})\Bigr) \;\cdot  ... \cdot \; \mu_d\Bigl( B_{d} \; ; U_{d}^{-1} (u_{d})\Bigr),
 \end{eqnarray*}
 where by \(\mu_j( \cdot ; t), j=1..d,\)  we denote the distribution of \[S_j(t) \eqd t^{1/ \alpha_j} S_j(1) + \left(t- t^{1/\alpha_j}\right)\mu_j. \]  
Therefore,  for \(B = B_{1} \times ... \times B_{d}\) defined above,
  \begin{eqnarray*}
  	   \label{nu1}
	 \nu(B) 
	 & =&
	   	\int_{\R_{+}} 
		\E_{\widetilde{F}(\cdot| U_d^{-1}(u_d))}
		 \Biggl[G(B_{1},..., B_{d-1})
		 \Biggr]
\mu_d\Bigl(  B_{d} \; ; U_{d}^{-1} (v) \Bigr) 
	 du_d,
 \end{eqnarray*} 
 where 
 \begin{eqnarray*} 
G(B_{1},..., B_{d-1}) &:=&  \mu_1\Bigl( B_{1} \; ; U_{1}^{-1} (\cdot) \Bigr) 
			\cdot
	...		\cdot
\mu_{d-1}\Bigl( B_{d-1} \; ; U_{d-1}^{-1} (\cdot) \Bigr),
 \end{eqnarray*} 
and  by \(\E_{\widetilde{F}(\cdot|v)}\) we denote the mathematical expectation with respect to the distribution with cdf  \(\widetilde{F}(\cdot|v)\). Due to  the Fubini theorem, 
\begin{eqnarray*}
 	\E_{{\widetilde{F}(\cdot|v)}}
		 \Biggl[
			G(B_{1},..., B_{d-1})
		 \Biggr]
	&=& 
		\int_{B_{1}} ... \int_{B_{d-1}} g(x_{1}, .., x_{d-1} \;| v) \; dx_{d-1} ... dx_{1},
\end{eqnarray*}
where
\begin{multline}
 \label{g}
g(x_{1}, .., x_{d-1} \; | v)  = \; 	\int_{\R_{+}^{d-1}}
			p_1 \Bigl( x_{1};
				U_{1}^{-1} (u_{1})
			\Bigr) \cdot
			... \cdot
			p_{d-1}\Bigl( x_{d-1};
				U_{d-1}^{-1} (u_{d-1})
			\Bigr) \\
			\cdot 
	 	 	\frac{
\partial^{d-1} \widetilde{F}(u_1,..., u_{d-1} | v) 
}{\partial u_{1} \; ...\; \partial u_{d-1}}
du_{1} ... du_{d-1}, \qquad v \geq 0,
\end{multline}
and \(p_j\left( \cdot \; ; t\right) \) is the  density of the measure  \(\mu_j( \cdot ; t), j=1..(d-1)\), which exists due to Proposition~3.12 from \cite{ContTankov}. Note that  \(g\) is a density function, see Remark~5.4 from \cite{panov2017a}. 
Therefore,  \eqref{nu} holds with 
\begin{eqnarray}\label{brevep}
\breve{p}_r(x_1, ..., x_d)  =g(x_1, ..., x_{d-1} | U_d^{-1}(r)) \cdot p_{d}\left( x_d; U_{d}^{-1} (r) \right).\end{eqnarray}
 On the next step, we aim to find a function \(\vH\) and a r.v. \(\vD\) such that \(\vH(r,\vD)\) has  density \(\breve{p}_r.\)

\textbf{3.} Changing the variables we get 
\begin{multline}
 \label{g}
g(x_{1}, .., x_{d-1} \; | v)   = \; 	\int_{\R_{+}^{d-1}}
			p_1 \Bigl( x_{1};
	y_1		\Bigr) \cdot
			... \cdot
			p_{d-1} \Bigl( x_{d-1};y_{d-1}	\Bigr) \\
\times
\left. 	\frac{\partial^{d-1}  \widetilde{F}(u_1,..., u_{d-1} | v) }{\partial u_{1} \; ...\; \partial u_{d-1}}				 \right|_{\substack {u_{1}= U_{1}(y_{1}) \\ ... \\ u_{d-1}=U_{d-1}(y_{d-1})}} \;
 U'_{1}(y_{1}) ... U'_{d-1}(y_{d-1})
 dy_{1}.. dy_{d-1}
\end{multline}
The last expression yields that \(g(\cdot | v) \) is in fact a pdf of the random vector
\begin{eqnarray*}
\left(  
	\zeta_{1} (\chi^{(1)})^{1/ \alpha_1}, ...,	\zeta_{d-1} (\chi^{(d-1)})^{1/ \alpha_{d-1}}
\right)
\end{eqnarray*}
where \(\zeta_{j}, j=1..(d-1)\) are independent r.v.'s with distribution \(S_{\alpha_j}(\beta_j, \sigma_j, 0)\) resp., 
and
\[ \left( \chi^{(1)}, ..., \chi^{(d-1)} \right)
\eqd 
\left( U_{1}^{-1}(\tilde\chi^{(1)}), ..., U_{d-1}^{-1}(\tilde\chi^{(d-1)})\right),
\] 
with \(\left(\tilde\chi^{(1)},..., \tilde\chi^{(d-1)}\right)\) having a distribution function 
\(\tilde{F}(\cdot |v)\).  Moreover, due to the assumption (A2),
\begin{eqnarray*}
\Bigl(\tilde\chi^{(1)},..., \tilde\chi^{(d-1)}
\Bigr)
\eqd 
\Bigl(
  h_1 (\xi_1, v), ..., h_{d-1} (\xi_{d-1}, v)
\Bigr). 
\end{eqnarray*}
Finally, we  conclude that the representation \eqref{nu} is fulfilled with 
\begin{eqnarray*}
\vH(r, \vD) &:=&  \left( 
\begin{matrix}
\zeta_{1} \left( U_{1}^{-1}(h_{1}(\xi_1, r))\right)^{1/ \alpha_1}\\
... \\
\zeta_{d-1} \left( U_{d-1}^{-1}(h_{d-1}(\xi_{d-1}, r))\right)^{1/ \alpha_{d-1}}\\
\zeta_{d} 	\left( U_{d}^{-1}(r) \right)^{1/ \alpha_d}
 \end{matrix}
 \right),
 \\
 \vD &:=& \left(
  \zeta_1,..., \zeta_d, \xi_1, ..., \xi_{d-1} 
\right) \in \R^{2d -1},
 \end{eqnarray*}
 where \(\zeta_{j}, j=1..d\) are independent r.v.'s with distribution \(S_{\alpha_j}(\beta_j, \sigma_j, 0)\).  
 
 \textbf{4.} To complete the proof, we should check that the drift under the
 choice of \(\breve{p}_r\) by \eqref{brevep}, coincides with the drift of the subordinated stable process, that is, 
 \begin{eqnarray}\label{check}
  \int_{\R_{+}^{d}}\eta(d \vec{y}) \int_{|\vx| \leq 1} \vx \; \mu(d\vx, \vec{y})= \int_{\R_+} \int_{|x| \leq 1} \vx \breve{p}_{r}(\vx) d\vx \; dr
.
\end{eqnarray}
Using the same techniques as on steps 2 and 3, we can represent the left-hand side in \eqref{check} as follows:
\begin{eqnarray*}
  \int_{\R_{+}^{d}}\eta(d \vec{y}) 
 \int_{|\vx| \leq 1} \vx \; \mu(d\vx, \vec{y})
&=&  \int_{\R_{+}^{d}} 
\left( 
\int_{|\vx| \leq 1} \vx \; p_1(x_1, y_1) ... p_d(x_d, y_d)\; d\vx
\right)\\
&&\hspace{0.1cm} \times
  \left.\frac{\partial^{d} F}{\partial u_{1} \; ...\; \partial u_{d}}\right|_{\substack {u_{1}= U_{1}(y_{1}) \\ ... \\ u_{d}=U_{d}(y_{d})}}
 U'_{1}(y_{1}) ... U'_{d}(y_{d})\;   dy_1
 ...
 dy_d\\
 &=& \int_{\R_{+}^{d}} 
\left( 
\int_{|\vx| \leq 1} \vx \; p_1(x_1, U_1^{-1}(u_1)) ... p_d(x_d,  U_d^{-1}(u_d))\; d\vx
\right)\\
&&\hspace{3cm} \times
\frac{\partial^{d} F (u_1,..,u_d)}{\partial u_{1} \; ...\; \partial u_{d}}du_1
 ...
 du_d
 \\
 &=&
 \int_{\R_+}
 \int_{|\vx| \leq 1} \vx\; p_d(x_d, U_d^{-1}(u_d))
% ...
\\
&&
\hspace{3cm}
 \times g(x_1, ..., x_{d-1} | U_d^{-1}(u_d)) 
 d\vx
  du_d\\
 &=& \int_{\R_+} \int_{|x| \leq 1} \vx \breve{p}_{r}(\vx) d\vx \;  dr.
  \end{eqnarray*}
  This observation completes the proof. 

\bibliographystyle{spbasic}
\bibliography{Panov_bibliography}

\end{document}